\documentclass[12pt]{amsart}

\usepackage{amsmath,amsthm,amssymb,mathrsfs,amsfonts,verbatim,enumitem,color}
\usepackage{mathabx}
\usepackage{etoolbox} 
\usepackage{bbm}
\usepackage[all,tips]{xy}
\usepackage{graphicx,ifpdf}
\ifpdf
\fi

\newtheorem{thm}{Theorem}[section]
\newtheorem{lem}[thm]{Lemma}

\newtheorem{prop}[thm]{Proposition}

\theoremstyle{definition}

\theoremstyle{remark}

\numberwithin{equation}{section}

\newcommand{\Rmnum}[1]{\expandafter\@slowromancap\romannumeral #1@}

\newcommand{\bG}{\mathbf{G}}
\newcommand{\bH}{\mathbf{H}}
\newcommand{\bV}{\mathbf{V}}
\newcommand{\bW}{\mathbf{W}}
\newcommand{\D}{\mathbb D}

\newcommand{\R}{\mathbb{R}}

\newcommand{\SL}{\operatorname{SL}}
\newcommand{\GL}{\operatorname{GL}}
\newcommand{\Z}{\mathbb{Z}}
\newcommand{\Q}{\mathbb{Q}}
\newcommand{\F}{\mathbb{F}}
\newcommand{\C}{\mathbb{C}}

\newcommand{\Ad}{\operatorname{Ad}}

\newcommand{\dd}{\; \mathrm{d}}
\newcommand{\df}{{\, \stackrel{\mathrm{def}}{=}\, }}

\begin{document}

\title{Measure Rigidity for solvable group actions in the space of lattices}

\author{Manfred Einsiedler}
\address{
Departement Mathematik, ETH Z\"urich, R\"amistrasse 101, 8092 Zurich, Switzerland}
\email{manfred.einsiedler@math.ethz.ch}

\author{Ronggang Shi}
\address{Shanghai Center for Mathematical Sciences, Fudan University, 220 Handan Road, Shanghai 200433, China}
 \email{ronggang@fudan.edu.cn}



\date{}



\begin{abstract}
{ We study invariant
probability measures}
on the  homogeneous space  $\SL_n(\R)/\SL_n(\Z)$ for  the
action of  subgroups of $\SL_n(\R)$ of  the form $SF$ where  $F$  is  generated by one parameter unipotent 
groups and $S$ is a one parameter $\R$-diagonalizable group normalizing $F$.
Under the assumption that $S$ contains an element with only one  eigenvalue   less  than one (counted with multiplicity) and others bigger  than one
 we prove that all the  $SF$ invariant and ergodic probability measures on $\SL_n(\R)/\SL_n(\Z)$ 
 are homogeneous.

\end{abstract}

\maketitle

\markright{}

\section{introduction}\label{sec;intro}
Let $n\ge 2$ be a positive integer. 
Let  $G=\SL_n(\R), \Gamma=\SL_n(\Z)$ and $X=G/\Gamma$.
Let  $S$ be a  one parameter $\R$-diagonalizable subgroup of $G$, i.e.~$S$ is
 conjugate in $G$ to 
 a subgroup  with diagonal matrices. 
  Let $F$ be a subgroup of $G$ normalized by $S$ and  generated by one parameter 
unipotent subgroups.
This paper is about the  classification of  $SF$ invariant and ergodic probability measures on $X$.

This type of questions has  already been  studied earlier by Margulis-Tomanov \cite{mt96} and Mozes \cite{m95} using Ratner's 
measure rigidity theorem \cite{r91} for  $F$ actions. In particular, it is proved in 
  \cite{mt96} and \cite{m95} that if  
  $SF$ is an epimorphic subgroup  of   $H\le G$   and $H$ is  generated 
  by one parameter  unipotent subgroups, then any $SF$ invariant and ergodic  probability measure is $H$ invariant and hence homogeneous 
by Ratner's theorem. Recently, in a joint paper with Barak Weiss the second named author  \cite{sw} proved 
  measure rigidity for certain non-epimorphic sovable subgroup actions on $X$ for   $n=3$. More precisely, it is proved
  in \cite[Theorem 5]{sw} that
  if $SF$ is nonabelian and $S$ has no nonzero fixed vectors in $\R^3$, then $SF$ acts uniquely ergodically on $X$. 
It is noticed already in \cite{sw} that even for $n=4$, there are homogeneous and even non-homogeneous $SF$ invariant and
ergodic probability measures under similar assumptions.   
Therefore, to get rigidity results one needs to put further restrictions on the group $S$. 

A one parameter $\R$-diagonalizable  subgroup of $G$ is said to be simple if 
it is  conjugate to   \[
\{\mathrm{diag}({e^{r_1 t}, \ldots, e^{r_{n-1}t}, e^{-t} ): t\in \R}
\}
\]
for a  probability vector $(r_1, \ldots, r_{n-1})\in \R^{n-1}$ with positive entries.
The following is the main result of this paper.

\begin{thm}\label{thm;rigidity}
	Let $F$ be a subgroup of $G=\SL _n(\R)$  $(n\ge 2)$ generated by one parameter 
	unipotent subgroups. 
Let $S=\{a_t: t\in \R \}$  be a simple  one parameter 
$\R$-diagonalizable
 subgroup of $G$  normalizing  $F$. 
 Then any $SF$ invariant and ergodic probability measure on $X=G/\Gamma$  $(\Gamma=\SL_n(\Z))$ is homogeneous.
 \end{thm}

The starting point of the  proof of Theorem \ref{thm;rigidity}  is the same as that of \cite{mt96} and \cite{m95}
where $\mu$ is written as a convex comination of $F$ invariant and ergodic probability measures.
Note that any  $F$ invariant and ergodic probability measure $\nu$ on $X$ is homogeneous by Ratner's theorem, i.e.~there 
is a closed subgroup $H$ of $G$ and $x_g\in X $ (here $x_g$ is the coset $g\Gamma$ for $g\in G$) such that $\nu$ is supported on $H x_g$ and is $H$ invariant. 
 We show that  $a_t \nu$  converges to zero in the space of finite measures on $X$ as $t$ tends to infinity or minus infinity unless $g^{-1}H g$ is a semisimple algebraic group virtually  defined over $\Q$ and  its action on $\R^n$ has
only one $\Q$-isotropic type. 
 The main new ingredient here  is the  use of  Tits \cite{t71} on  the structure 
of irreducible  representations to show that the  centralizer of $H$ in $G$ (denoted by $Z_G(H)$) is an abelian group. The proof is then completed by a standard argument.

\subsubsection*{Acknowledgements}
The authors would like to thank MSRI 
 for its hospitality during Spring 2015.

\section{Irreducible representations over $\Q$}

This section is about  the structure of  $\Q$-rational representations of semisimple algebraic groups defined over $\Q$. 
We adopt the convention of \cite{b} to identify algebraic varieties or groups  with their rational points over complex 
numbers and denote  them by boldfaced capital  letters.

Let $\bV$  be an $n$-dimensional vector space over $\C $ endowed with a $\Q$-structure. 
Let $\bG=\SL (\bV)$ be the special linear group with the natural $\Q$-structure 
such that  linear action of $\bG $ on $\bV$ is a $\Q$-rational representation. 
Let 
 $\bH$ be a nontrivial  connected   semisimple  subgroup of $\bG$ defined over  $\Q$. 
 Recall that   the  linear action of $\bH$ on $\bV$ is completely determined by the 
 induced    Lie algebra representations which   is completely reducible over $\Q$ (see \cite[Theorem 7.8.11]{weible}).
So $\bV$ is a direct sum of $\bH$ invariant  $\Q$-irreducible subspaces. 
We say the linear action of $\bH$ on $\bV$ has only one $\Q$-isotropic type if all these irreducible subspaces 
are isomorphic to each other over $\Q$.

 For a field $\mathbb F$ containing $\Q$ and a variety $\mathbf X$ defined over $\Q$, we let 
 $\mathbf X_{\mathbb F}$ be the set of $\mathbb F$-rational points of $\mathbf X$. 
 Let $V=\bV_\R$, $G=\bG_\R$ and $H=\bH_\R$. 
 The main result of this section is the following proposition.
 \begin{prop}
 	\label{lem;basic}
 	Let $S$  be a simple  one parameter $\R$-diagonalizable
 	subgroup of $G$. Suppose 
 $S$ is contained in $HZ_G(H)$ 
 		and  the linear action of $\bH$ on $\bV$ has only one $\Q$-isotropic type.
Then $Z_G(H)$ is an abelian  group. 
 \end{prop}

\begin{proof}

Since the linear action of $\bH$ on $\bV$ has only one $\Q$-isotropic type, 
there
exists  a positive integer $m$ such that 
$\bV$ is isomorphic to  the direct sum of $m$ copies of a certain $\Q$-irreducible  representation of $\bH$. 
According to \cite[Theorem 7.2]{t71}, there exists a central simple  division algebra $\D$ over a number field $\F$,
a complex   vector space  $\bW$  defined over  $\F$ with $\bW_\F$ having  a right $\D$-module structure and 
an absolutely
irreducible $(\D, \F)$-representation   $\rho: \bH\to \GL(\bW)$ such that 
\begin{align}\label{eq;tits}
\bV\cong \mathrm{Res}_{\F/\Q} \big(\mathrm{Res}_{\D/\F} \bW^{\oplus m}\big)
\end{align}
 in the category of $\Q$-rational representations of $\bH$. 
According to Tits' notation, the morphism $\rho$ is defined over $\F$
and 
$\rho(\bH_\F)$ acts $\D$-linear on $\bW_\F$. 
The notion of  absolute irreducibility will be explained later. 

Let  $d^2=[\D: \F]$ and  $k=[\bW_
\F:\D]$. 
We will  show that $m=d=1$, which will imply the proposition immediately. 
We would like to understand the group    $H$ and its centralizer
by choosing a good basis of $V$ using   (\ref{eq;tits}). 	
	To make the presentation simpler we assume the left and right hand sides of (\ref{eq;tits}) are equal. 
We will give explicit descriptions  of $\D, \bW$  and of restriction of scalars using several  isomorphisms. 

By choosing an isomorphism  $\D\otimes_\F \C\cong \mathrm{M}_d(\C) $  and identifying $\D$ with $\D\otimes 1$ in the natural 
way we assume that $\D$ is a  subring of $\mathrm{M}_d(\C)$.  
 We assume $\bW_\F= \D^k$ and $\bW= \mathrm{M}_d(\C)^k$ 
where $\mathrm{M}_d(\C)^k$ is naturally identified with the set of  $kd$-by-$d$ matrices  by viewing its elements blockwise. 
Therefore, the group  
of right $\mathrm{ M}_d(\C)$-module automorphisms of $\bW$
can be naturally identified with   $\GL_{kd}(\C) $ which acts on $\bW$ from the left. Let $\widehat{\rho}$ be the composite of $\rho$ and this identification, that is, 
$\widehat{\rho}(\bH)$ is a subgroup of $\GL_{kd}(\C) $. We consider  $\widehat\rho$ as  a
 morphism of algebraic groups defined over $\F$ with 
\[
\widehat\rho(\bH_\F)\le \GL_k(\D). 
\]
We note that  the  absolute irreducibility of  $\rho$ refers to the fact  that
$\widehat\rho(\bH)$ acts irreducibly on   $\C^{kd}$.

Let  $\sigma_1, \ldots, \sigma_r $ be the  real embeddings of $\F$ and 
$\sigma_{r+1}, \sigma_{r+1}'~, \ldots, \sigma_{r+s}, \sigma_{r+s}'$ be the  pairs of complex embeddings of $\F$. 
Each $\sigma_i$ defines a functor from the category of varieties defined over $\F$ to the category of  varieties defined over 
$\F^{\sigma_i}\df\{\sigma_i(x): x\in \F  \}$ (see \cite[\S 0.14.1]{b} for definitions).
For  a variety $\mathbf X$ defined over $\F$ we let $\mathbf X^{\sigma_i}$ be the image of this functor. 
Also let $\mathbb D^{\sigma_i}=\mathbb D\otimes _{\mathbb F}\mathbb F^{\sigma_i}$ where the left $\mathbb F$-module structure of
  $\mathbb F^{\sigma_i}$ comes from the left multiplication of $\sigma_i(\cdot)$.

According to  (\ref{eq;tits}) and \cite[Theorem 1.3.2]{weil} we have 
\[
V=\prod_{i=1}^{r+s} V_i\quad \mbox{where}\quad
V_i=\left\{ 
\begin{array}{ll}
\big ({\bW}^{\sigma_i}_{\R}\big )^{\oplus m}  & 1\le i\le r\\
\big ({\bW}^{\sigma_i}\big )^{\oplus m}&  r+1\le i\le r+s
\end{array}
\right .
\]
and 
\begin{align}\label{eq;basicH}
H=\prod_{i=1}^{r+s} H_i\quad \mbox{where}\quad
H_i=\left\{ 
\begin{array}{ll}
\rho^{\oplus m}(\bH)^{\sigma_i}_\R & 1\le i\le r\\
\rho^{\oplus m}(\bH)^{\sigma_i}& r+1\le i\le r+s.
\end{array}
\right .
\end{align}
Moreover, 
\[
Z_G(H)=G\cap \prod_{i=1}^{r+s}Z_i,
\]
where 
 $Z_i$ is the centralizer of $H_i$ in $\GL_\R (V_i)$. It suffices to prove that all the  $Z_i$  are abelian groups.

For $r+1\le i\le r+s$, 
the complex part of the above decomposition is easy to understand.  
Since $\mathbb D^{\sigma_i}\otimes _{\mathbb F^{\sigma_i}}\C=\mathrm{M}_d(\C)$,
we have a decomposition  
 \[ 
 V_i=
V_i'\otimes_\C V_i
 ''\] where   $V_i' = \C^{kd}$ and $V_i''=  \C^{md}$
  such that  the group  
\begin{align}\label{eq;hic}
H_i=\widehat\rho(\bH)^{\sigma_i}\otimes \mathrm{I}_{V_i''}
\end{align}
where $ \mathrm{I}_{V_i''}$ denotes the identity map  of $V_i''$. 

For the real part we first note that 
the central simple algebra  $\D^{\sigma_i}\otimes_{{\F}^{\sigma_i}} \R $ is equal  to $\mathrm{M}_d(\R) $ or 
$\mathrm{M}_{d/2}(\mathbb H) $ where in the latter case $\mathbb H $ is the  Hamilton quaternions and  $d$ has to be  an odd number. 
By choosing an embedding $\mathbb H \to \mathrm{M}_2(\C)$ we assume that $\mathrm{M}_{d/2}(\mathbb H)$ is 
a subring of $\mathrm{M}_{d}(\C)$.   Suppose for $1\le i\le l$ we are in the case of $\mathrm{M}_d(\R)$ and for $l+1\le i\le r$ we are in the case
of $\mathrm{M}_{d/2}(\mathbb H) $. 
Then 
\begin{align}\label{eq;hir}
V_i=V_i'\otimes_\R V_i'' \quad\mbox{and}\quad
H_i =\widehat\rho(\bH)^{\sigma_i}_\R\otimes \mathrm{I}_{V_i''}
\end{align}
where
\begin{align*}
V_i'& =\R^{kd} \mbox{ and } V_i''= \R^{md} \quad \mbox{ for }1\le i\le l;\\
V_i'& =\mathbb H^{kd/2} \mbox{ and }   V_i''= \R^{md/2} \quad \mbox{ for }l+1\le i\le r.
\end{align*}

{
For $1\le i\le r$ (resp.~$r+1\le i\le r+s$) let   $R_i $   be the ring of $\R$-linear spans of $\widehat \rho(\bH)^{\sigma_i}_\R$ (resp.~$\widehat \rho(\bH)^{\sigma_i}$) in the ring of $\R$-linear transformations of $V_i'$. }
 Since $\widehat\rho(\bH)$ acts irreducibly on $\C^{kd}$, the space  $V_i'$ is an irreducible $R_i$-module. 
 For every $ i$ there is a division algebra $D_i$ such that the opposite division algebra $D_i^\circ$
is isomorphic to  set of $R_i$-module morphisms of $V_i'$.
By Wedderburn-Artin theorem (see e.g.~\cite[\S 1]{farb}),
 the ring  $R_i$ is isomorphic to 
a matrix ring over $D_i$  and  $V_i$ is isomorphic to a free $D_i$ module of finite rank. 
Since $\widehat\rho(\bH)$ acts irreducibly on $\C^{kd}$, we have 
\begin{equation}\label{eq;di}
D_i=\left  \{
\begin{array}{ll}
\R & 1\le i\le l \\
 \mathbb H &  l+1\le i\le r\\
\C & r+1\le i\le r+s.
\end{array}
 \right . 
\end{equation}

We can understand $Z_i$ using 
 (\ref{eq;di}), (\ref{eq;hic}) and (\ref{eq;hir}). 
For $r+1\le i\le r+s$, 
\[
Z_i= (\C \mathrm{I}_{V_i'}\otimes_\C \mbox{Hom}_\C(V_i'', V_i''))^*=\mathrm{I}_{V_i'}\otimes \GL_\C (V_i'')
\]
where $(\cdot)^*$ denotes the invertible elements of a ring. 
Moreover, if $c_i\in Z_i$ is diagonalizable over $\R$, then we have a decomposition  
$c_i=\mathrm{I}_{V_i'}\otimes b_i$ where $b_i$ is an $\R$-diagonal element of  $\GL_\C(V_i'')$
(here we view $\GL_\C (V_i'')$ as the real points of  an algebraic group coming from  the    restriction of scalars).
  For  $1\le i\le l$ 
\[
Z_i=\big(\R \mathrm{I}_{V_i'}\otimes_\R  \mbox{Hom}_\R (V_i'', V_i'')\big)^*=\mathrm{I}_{V_i'}\otimes \GL_\R(V_i''),
\] 
and  every $\R$-diagonalizable    $c_i\in Z_i$  can be decomposed as 
\[
\mathrm{I}_{V_i'}\otimes b_i
\]
where $b_i\in \GL_\R(V_i'')$ is  $\R$-diagonalizable. 
 For $l+1\le i\le r$  
\[
Z_i=\big(\mathbb H \mathrm{I}_{V_i'} \otimes _\R \mbox{Hom}_\R(V_i'')\big)^*\cong   \GL_{md/2}(\mathbb H \mathrm{I}_{V_i'}),
\]
where $\mathbb H \mathrm{I}_{V_i'}$ acts on $V_i'$ via right multiplications of matrices. 
Suppose    $c_i\in Z_i$ is $\R$-diagonalizable, then there exists $g_i\in  Z_i$   such that 
\[
g_ic_ig^{-1}_i=\mathrm{I}_{V_i'}\otimes b_i
\]
where $b_i\in \GL_\R (V_i'')$ is  $\R$-diagonalizable.

Now we turn to the proof of the conclusion. 
Since $S$ is contained in $HZ_G(H)$, the group $S$ is the direct product 
$S_1S_2$ of one parameter  $\R$-diagonalizable subgroups $S_1\le H$ and 
$S_2\le Z_G(H)$. 
Since $S$ is simple, there exists $a\in S$ such that the linear action of 
$a$ on $\R^n$ has only one eigenvalue $\lambda<1$ counted with multiplicity and all the other eigenvalues are bigger than $1$.  
We write $a=hb$ where $h\in S_1$ and $b\in S_2$.
 According to  the  discussions previously, by possibly replacing $
 S $ by $gSg^{-1}$ for some $g\in \prod_{i=1}^{r+s}Z_i$  we have
 \begin{align}\label{eq;basic1}
h=\prod_{i=1}^{r+s}h_i\otimes \mathrm {I}_{V_i''} \quad \mbox{and}\quad  b= \prod_{i=1}^{r+s} \mathrm{I}_{V_i'}\otimes b_i
 \end{align}
 where  all the $h_i$ and   $b_i$ are 
 $\R$-diagonalizable elements  of  $\GL_\R(V_i')$ and  $\GL_\R (V_i'')$ respectively.

 There exists $i_0$ with $1\le i_0\le r+s$ such that 
 $\lambda$ is an eigenvalue of $a$ restricted to  $V_{i_0}$. 
Since   each eigenvalue of $a$ restricted to  $V_i$
has multiplicity at least  $4$ (resp.~$2$)      for $l<i\le r$ (resp.~$r+1\le i\le r+s$), we have $i_0\le l$.  
We assume without loss of generality that $i_0=1$.   
Since $\det h_i=1$ for all $i$,  $\lambda$ has multiplicity $1$, all the other eigenvalues 
of $a$ are bigger than $1$ and (\ref{eq;basic1}), we have $\det b_i>1$ for $i>1$. Since $\det a=1$ we have   $0<\det b_1\leq 1$ ({with equality
in case~$\mathbb{F}=\mathbb{Q}$}).
This and the  simplicity of $\lambda $ implies that $\min\{kd,md\}=1$.  Therefore   $d=1$.   
Since $\bH$ is a semisimple algebraic group, we have $k> 1$ and hence $m=1$. 
Using previously discussed structures of $Z_G(H)$ and $m=d=1$, we have $l=r $ and $Z_G(H)$ is an abelian group.

\end{proof}

\section{measure rigidity}

 The aim of this section is to prove Theorem \ref{thm;rigidity}.
 Let $G, \Gamma, X, S, F$ be as in Theorem \ref{thm;rigidity}.
 For every $g\in G$ we use 
 $x_g$  to denote the coset $g\Gamma\in X$.
Also, we use $x_\Gamma$ to denote the coset  $\Gamma\in X$. 
We endow $G$ and    $\R^n$  with    $\Q$-structures such that their  $\Q$-rational points are $\SL_n(\Q)$ and  $\Q^n$
respectively.  A subgroup $H$  of $G$ is said to be virtually  defined over $\Q$ if it has finite index in the group of  real points of 
a $\Q$-rational algebraic group.
The following is a version of Ratner's theorem \cite{r91}, see also
the work \cite[Prop.~1.1]{Borel-Prasad} of Borel-Prasad
and \cite[Theorem 2]{t00} of Tomanov.  
 \begin{thm}\label{lem;structure}
  Let $\nu$ be an $F$ invariant and ergodic probability measure on $X$. 
  Then there exists a connected subgroup $H$
 of $G$ virtually  defined over $\Q$  with semisimple Levi factor
 and a point $x_g\in X$ such that
 the measure $\nu$ is  invariant under the action of $gH g^{-1}\ge F$   and
 $\nu$ is concentrated on  the closed subset 
  $gH x_\Gamma$.    
 \end{thm}

 \begin{lem}\label{lem;observe}
Suppose   there is a proper $\Q$-rational subspace $V$ of $\R^n$ invariant under $SH$ where 
$H$ is a subgroup of $G$ normalized by $S$ . 
 Then either $\{a_thx_\Gamma: t\ge 0\}$ is divergent  for every $h\in H$ or 
  $\{a_thx_\Gamma: t\le 0\}$ is divergent  for every $h\in H$.
  \end{lem}
  \begin{proof}
 By assumption, there exists an integer $k$ with  $1\le k\le n-1$ and 
 a basis $v_1, \ldots, v_k$ of $V$  consisting of rational vectors. 
Since the   space $V$ is $SH$ invariant, the vector
 $w=v_1\wedge\cdots \wedge v_k\in \wedge^{k}\Q^n$ is an eigenvector of $S$.
By assumption the one parameter group  $S$ is simple, so   there exists a nonzero real number $\lambda$
 such that $a_t w=e^{\lambda t}w$.   Therefore
 $a_t w\to 0$ as $t\to \infty$ or $t\to -\infty$. We assume without loss of generality 
 that $a_t  w\to 0$ as $t\to\infty$. 
 Note that for any $h\in H$ one has
 $h w=s w$ for some nonzero real number  $s$,  which implies that 
 $
 a_thw\to 0 
 $ as 
 $t\to \infty$. 
   It follows from the Mahler's compactness criterion (see e.g.~\cite[Propositon 3.1]{w04}) that the trajectory $\{a_thx_\Gamma: t\ge 0\}$ is 
   divergent for every $h\in H$. 
\end{proof}
\begin{proof}[Proof of Theorem \ref{thm;rigidity}]

Let $\mu$ be an $SF$ invariant and ergodic probability measure on $X$. 
It follows from Poincare recurrence that for almost every $x\in X$
\begin{enumerate}
\item 	 Neither of the trajectories  $\{a_tx: t\ge 0\}$  or $\{a_tx: t\le 0\}$ is divergent. 	
\end{enumerate}

Let 
\[
\mu=\int_X\mu_x\dd\mu(x)
\]
be the ergodic decomposition of $\mu$ into $F$ invariant and  ergodic components. 
According to  \cite{mt96} and \cite{m95} we can find a connected
subgroup $H$ of $G$ and   a measurable subset $X'$ of $X$ with full $\mu$ measure such that for every
$x\in X'$ the following properties  hold:
\begin{enumerate}
\item[(2)]  $\mu_x$ is  $H$ invariant and supported on the closed subset $Hx$. 
\item[(3)]  $S\le N_G^1(H)^{\circ}$  where $N_G^1(H)$ consists of    $g\in N_G(H)$ with the property that 
the conjugation by $g$ preserves the Haar measure of $H$. Moreover, $\mu(N_G^1(H)^\circ x)=1$. 
\end{enumerate}
By possibly removing a  measure zero  subset from $X'$, we assume  that (1) holds for all $x\in X'$. 

According to (2) and  Theorem \ref{lem;structure} there is a connected   subgroup $H'$ of $G$ virtually  defined over $\Q$
with semisimple Levi factor  and $g\in G$ with $x_g\in X'$  such that 
$H=g H' g^{-1}$.  By replacing $S$ by $g^{-1}Sg$, $F$ by $g^{-1}Fg$
and $\mu $ by $g^{-1}\mu$ we can without loss of generality assume that $H=H'$  and $x_\Gamma\in X'$.

Let $\mathscr R_u(H)$ be the maximal  unipotent normal subgroup of $H$. Then $H$ has  a connected semisimple subgroup 
$H_0$ virtually defined over $\Q$ such that $H=H_0\ltimes \mathscr R_u(H)$.
Note that 
 $\Gamma\cap H_0$ and $\Gamma \cap \mathscr R_u(H)$ are lattices in  $H_0$
 and $ \mathscr R_u(H)$ respectively.  
 As a connected unipotent group, the space  of 
$\mathscr R_u(H)$ invariant vectors 
\[
V\df \{v\in \R^n: h v=v \mbox{ for all } h\in \mathscr R_u(H)\}
\]   
is nonzero and defined over $\Q$. Since $S$ normalizes $H$, the space $V$ is 
also  $S$-invariant.  
If $\mathscr R_u(H)$ is nontrivial, then $V$ is a proper subspace of $\R^n$. Hence
Lemma \ref{lem;observe} implies that either 
$\{a_tx_\Gamma: t\ge 0\}$ or $\{a_tx_\Gamma: t\le 0\}$ is divergent. This contradicts  (1). 
Therefore $\mathscr R_u(H)$ is trivial and $H=H_0$ is semisimple.

Let $Z=Z_G(H)$. 
Note that $N_G(H)$ and $HZ$ have the same connected component, see \cite{p98}.
So by (3) we have $S\le HZ$. 
  The natural representation of $H$ on $\R^n$ is defined over $\Q$ and is completely 
reducible over $\mathbb Q$. 
Since $S$ is contained in $HZ$, every maximal $\Q$-rational H-invariant  subspace which  has only one $\Q$-isotropic type
 is $S$-invariant.  
By Lemma \ref{lem;observe} and (1) the representation of $H$ on $\R^n$ has only one $\Q$-isotropic type.
Therefore Proposition  \ref{lem;basic} implies that $Z$ is  an abelian group. 
In view of  (3) we get $\mu(Z Hx)=1$.
By pulling $\mu$ back to $Z H/\Gamma_1$ where 
$\Gamma_1=ZH\cap \Gamma$ we get an 
$SF$-invariant and ergodic  probability measure $\mu_1$  on $Z H/\Gamma_1$.
 Note  that $SF$ is generated by Ad-unipotent one parameter subgroups of $ZH$. 
It follows from Ratner's theorem 
that 
$\mu_1$ is a homogeneous measure. Therefore $\mu$ is homogeneous.

\end{proof}


\begin{thebibliography}{99}

\bibitem{b}
A. Borel, Linear algebraic groups, 2nd ed., GTM 126, Springer-Verlag, 1991. 

\bibitem{Borel-Prasad}
A.~Borel and G.~ Prasad.
\newblock \emph{Values of isotropic quadratic forms at S-integral points}. 
\newblock Compositio Math. 83 (1992), no. 3, 347–372.


\bibitem{farb}
B. Farb and K. Dennis, Noncommutative Algebra, GTM 144, Springer-Verlag, 1993. 


\bibitem{mt96}
G. A. Margulis, G. M. Tomanov, 
Measure rigidity for almost linear groups and its applications,  J. Anal. Math. 69 (1996), 25-54.




\bibitem{m95}
S. Mozes, 
Epimorphic subgroups and invariant measures, 
Ergodic Theory Dynam. Systems 15 (1995), no. 6, 1207-1210.



\bibitem{p98}
D. Poguntke, Normalizers and centralizers of reductive
 subgroups of almost connected Lie groups, 
J. Lie Theory 8 (1998), no. 2, 211-217. 



\bibitem{r91}
M. Ratner, On Raghunathan's measure conjecture. Ann. of Math. (2) 134 (1991), no. 3, 545-607.

\bibitem{sw}
R. Shi and B. Weiss, Invariant measures for solvable groups and Diophantine approximation, to appear in Israel J. Math..


\bibitem{t71}
J. Tits, Repr\'esentations lin\'eaires irr\'eductibles
 d'un groupe r\'eductif sur un corps quelconque.
  (French) J. Reine Angew. Math. { 247} (1971), 196-220. 

\bibitem{t00}
G.   M. Tomanov, 
Orbits on homogeneous spaces of arithmetic 
origin and approximations, in Analysis on homogeneous
 spaces and representation theory of Lie groups, 
 Okayama-Kyoto (1997), 265-297, Adv. Stud. Pure Math. 26, Math. Soc. Japan, Tokyo, 2000.
 
\bibitem{weible}
C. A. Weibel, An introduction to homological algebra, 
 Cambridge Studies in Advanced Mathematics 38, Cambridge University Press, Cambridge, 1994.
 
 \bibitem{weil}
A. Weil, Adeles and algebraic groups,
 Progress in Mathematics 23,  Birkhauser, 1982.

\bibitem{w04}
B. Weiss, Divergent trajectories on noncompact parameter spaces, Geom.
Funct. Anal. 14 (2004), no. 1, 94-149.






\end{thebibliography}
\end{document}